\documentclass[11pt]{article}
\usepackage{epsfig}
\usepackage{graphicx}
\usepackage{setspace}
\usepackage{color}
\usepackage{amsmath,amsthm,amssymb}
\usepackage{latexsym}
\newtheorem{theor}{Theorem}[section]
\newtheorem{theo}[theor]{Theorem}

\newtheorem{lemma}[theor]{Lemma}

\usepackage{hhline}

  \baselineskip 8 mm

\begin{document}

\title{Cubicity, Boxicity, and Vertex Cover}

\author{L. Sunil Chandran \thanks{Computer Science and Automation department,
Indian Institute of Science,
Bangalore- 560012,
India.  {\tt sunil@csa.iisc.ernet.in}} \and Anita Das
 \thanks{Computer Science and 
Automation department,
Indian Institute of Science,
Bangalore- 560012,
India.  {\tt anita@csa.iisc.ernet.in}}
\and Chintan Shah  
 \thanks{Computer Science and 
Automation department,
Indian Institute of Science,
Bangalore- 560012,
India.  {\tt chintan@csa.iisc.ernet.in}}
}

\date{}
\maketitle

\begin{abstract}
A $k$-dimensional box is the cartesian product $R_1 \times R_2
\times \cdots \times R_k$ where each $R_i$ is a closed interval on
the real line. The {\it boxicity} of a graph $G$, denoted as
$box(G)$, is the minimum integer $k$ such that $G$ is the
intersection graph of a collection of $k$-dimensional boxes. A unit
cube in $k$-dimensional space or a $k$-cube is defined as the
cartesian product $R_1 \times R_2 \times \cdots \times R_k$ where
each $R_i$ is a closed interval on the real line of the form $[a_i,
a_{i}+1]$. The {\it cubicity} of $G$, denoted as $cub(G)$, is the
minimum $k$ such that $G$ is the intersection graph of a collection
of $k$-cubes. In this paper we show that $cub(G) \leq t + \left
\lceil \log \ (n - t)\right\rceil - 1$ and $box(G) \leq \left
\lfloor\frac{t}{2}\right\rfloor + 1$, where $t$ is the cardinality
of the minimum vertex cover of $G$ and $n$ is the number of vertices
of $G$. We also show the tightness of these upper bounds.

F. S. Roberts in his pioneering paper on boxicity and cubicity had
shown that for a graph $G$, $box(G) \leq \left \lfloor\frac{n}{2}
\right \rfloor$, where $n$ is the number of vertices of $G$, and
this bound is tight. We show that if $G$ is a bipartite graph then
$box(G) \leq \left \lceil\frac{n}{4} \right\rceil$ and this bound is
tight. We point out that there exist graphs of very high boxicity
but with very low chromatic number. For example there exist
bipartite (i.e., 2 colorable) graphs with boxicity equal to
$\frac{n}{4}$. Interestingly, if boxicity is very close to
$\frac{n}{2}$, then chromatic number also has to be very high. In
particular, we show that if $box(G) = \frac{n}{2} - s$, $s \geq 0$,
then $\chi(G) \geq
\frac{n}{2s+2}$, where $\chi(G)$ is the chromatic number of $G$.\\

\noindent {\bf Key words:} Boxicity, cubicity, vertex cover.
\end{abstract}

%For a bipartite graph $G = (V_1\cup V_2, E)$, we show that $box(G)
%\leq$ min$\{ \lceil\frac{n_1}{2}\rceil,
%\lceil\frac{n_2}{2}\rceil\}$, where $n_1 = |V_1|$ and $n_2 = |V_2|$.
%We also show that this upper bound for bipartite graphs is tight.
%Finally, we compared the parameters chromatic number and boxicity of
%a graph. We show that when boxicity of a graph is very high (near to
%$\frac{n}{2}$), then chromatic number also increases highly (near to
%$\frac{n}{2}$).

%The best known upper bound for cubicity of a graph $G$ is known to
%be $\frac{2n}{3}$. In this paper, we have shown that $cub(G) \leq
%|vc| - 1 + log(n - \alpha(G))$, where $\alpha(G)$ is the
%independence number of $G$ and also have shown the tightness of this
%bound.

\section{Introduction}

Let $\mathcal{F}$ be a family of non-empty sets. An undirected graph
$G$ is an intersection graph for $\mathcal{F}$ if there exists a
one-one correspondence between the vertices of $G$ and the sets in
$\mathcal{F}$ such that two vertices in $G$ are adjacent if and only
if the corresponding sets have non-empty intersection. If
$\mathcal{F}$ is a family of intervals on real line, then $G$ is
called an {\it interval graph}. If $\mathcal{F}$ is a family of
intervals on real line such that all the intervals are of equal
length, then $G$ is called a {\it unit interval graph}.

A $k$-dimensional box or $k$-box is the cartesian product $R_1
\times R_2 \times \cdots \times R_k$, where each $R_i$ is a closed
interval on the real line. The boxicity of a graph $G$ is defined to
be the minimum integer $k$ such that $G$ is the intersection graph
of a collection of $k$-boxes. Since $1$-boxes are nothing but closed
intervals on the real line, interval graphs are the graphs having
boxicity $1$.

A unit cube in $k$-dimensional space or a $k$-cube is defined as the
cartesian product $R_1 \times R_2 \times \cdots \times R_k$ where
each $R_i$ is a closed interval on the real line of the form $[a_i,
a_{i}+1]$. A $k$-cube representation of a graph is a mapping of the
vertices of $G$ to $k$-cubes such that two vertices in $G$ are
adjacent if and only if their corresponding $k$-cubes have a
non-empty intersection. The {\it cubicity} of $G$ is the minimum $k$
such that $G$ has a $k$-cube representation. Note that a $k$-cube
representation of $G$ using cubes with unit side length is
equivalent to a $k$-cube representation where the cubes have side
length $c$ for some fixed positive number $c$. The graphs of
cubicity $1$ are exactly the class of unit interval graphs. Clearly
$box(G) \leq cub(G)$.

The concept of boxicity and cubicity was introduced by F. S. Roberts
\cite{roberts69} in 1969. Boxicity finds applications in fields such
as ecology and operations research. Computing the boxicity of a
graph was shown to be NP-hard by Cozzens \cite{cozzens81}. This was
later strengthened by Yannakakis \cite{yannakakis}, and finally by
Kratochvil \cite{kratochvil94} who showed that deciding whether
boxicity of a graph is at most two itself is NP-complete. It has
been shown that deciding whether the cubicity of a given graph is at
least three is NP-hard \cite{yannakakis}.

%The complexity of finding the maximum independent set in bounded
%boxicity graphs was considered in \cite{takao, fowler}.

Recently many new upper bounds have been derived for boxicity. In
\cite{sunil}, it is shown that $box(G) \leq 2 \Delta^2$, where
$\Delta$ is the maximum degree of the graph $G$. It is shown in
\cite{boxicitytreewidth} that $box(G) \leq tw(G) + 2$, where $tw(G)$
is the treewidth of $G$. In \cite{sunil2}, it is shown that $box(G)
\leq (\Delta + 2) \log \ n$, where $n$ is the number of vertices of
the graph $G$.

There have been many attempts to bound the boxicity of graph classes
with special structure. F. S. Roberts \cite{roberts69} proved that
the boxicity of a complete $k$-partite graph is $k$. Scheinerman
\cite{Scheinerman} showed that boxicity of outer planar graphs is at
most two. Thomassen \cite{thomassen} proved that the boxicity of
planar graphs is bounded above by three. The boxicity of split
graphs is investigated by Cozzens and Roberts \cite{cozzenrobert}.
Upper bounds on the boxicity of some special classes of graphs such
as chordal graphs, circular arc graphs, AT-free graphs, permutation
graphs, co-comparability graphs are given in
\cite{boxicitytreewidth}. The cube representation of special classes
of graphs like hypercubes and complete multipartite graphs were
investigated in \cite{cub1, cub2, maehara, quint, roberts69}.

\subsection{Our results}

A {\it vertex cover} of $G$ is a set $Q \subseteq V(G)$ that
contains at least one endpoint of every edge of $G$. Among all
vertex covers of $G$, the minimum cardinality vertex cover is called
a {\it minimum vertex cover} of $G$ and is denoted by $MVC$. A set
$A \subseteq V$ is called an {\it independent set} if the vertices
in $A$ are pairwise non-adjacent. Vertex cover is a central
parameter in graph theory and computer science. In fact it is one of
the earliest parameters to be studied in graph theory: K\"onig's
Theorem (1931) states that in a bipartite graph the cardinality of a
maximum matching is equal to the cardinality of a minimum vertex
cover. The vertex cover problem was one of Karp's 21 NP-complete
problems. It is easy to see that if $MVC$ is a minimum vertex cover
of $G$ then $V - MVC$ is a maximum independent set of $G$.

In this paper we relate the concept of vertex cover with boxicity
and cubicity. In particular we show the following:\\

\noindent{\bf Result 1.} $cub(G) \leq t + \left\lceil\log \ (n -
t)\right\rceil - 1$, where $t$ is the cardinality of a minimum
vertex cover of $G$, and this
upper bound is tight.\\

\noindent{\bf Result 2.} $box(G) \leq
\left\lfloor\frac{t}{2}\right\rfloor + 1$, where $t$ is the
cardinality of a minimum vertex cover of $G$,
and this upper bound is tight.\\

\noindent{\bf Remark 1:} It was shown in \cite{boxicitytreewidth}
that $box(G) \leq tw(G) + 2$, where $tw(G)$ is the treewidth of the
graph $G$. It can be shown that $tw(G) \le t$, where $t$ is the
cardinality of a minimum vertex cover of $G$. From this we can 
infer that $box(G) \le t+2$.  But the inequality $tw(G) \leq
t$ is tight (for example Roberts Graphs, or complete graphs). Moreover, the
inequality $box(G) \leq tw(G) + 2$ is shown to be tight up to an
additive lower order factor \cite{boxicitytreewidth}. Therefore, it
is not possible (by strengthening this approach) to get an upper bound for 
boxicity in terms of $t$
comparable to what is shown in
this paper.\\

\noindent{\bf Remark 2:} Let $\mathcal{M}_{G}$ denote the set of all
maximal matchings of $G$. Let $\nu(G) = \min_{M \in
\displaystyle\mathcal{M}_{G}}|M|$, i.e., the cardinality of the
minimum maximal matching in $G$. It was shown in \cite{cozzenrobert}
that $box(G) \leq t'(\overline{G})$, where $\overline{G}$ is the
complement of $G$ and $t'(\overline{G})$ is the minimum number of
edges of $\overline{G}$ which are incident to all the edges of
$\overline{G}$. It is easy to verify that 
$t'(\overline{G}) = \nu(\overline{G})$. Also, as $t \leq
2\nu(G)$, by
Result 2 it follows that, $box(G) \leq \nu(G)+1$. So, by combining
Result 2 and the result due to Cozzens {\sl et al.}
\cite{cozzenrobert}, we infer
that, $box(G) \leq \min\{\nu(G)+1, \ \nu(\overline{G})\}$.\\

\noindent{\bf Result 3.} For a bipartite graph $G = (V_1 \cup V_2,
E)$, $box(G) \leq \min\{\lceil\frac{n_1}{2}\rceil,
\lceil\frac{n_2}{2}\rceil\}$, where $n_1 = |V_1|$ and $n_2 =
|V_2|$. This upper bound is tight.\\

\noindent {\bf Remark 3:} The above upper bound for bipartite graphs should be 
 compared with the upper bound for general graphs given by F. S. Roberts in his
 pioneering paper \cite {roberts69}, namely $box(G) \le \left \lfloor \frac {n}{2} \right \rfloor$
 where $n$ is the number of vertices in $G$.  \\

\noindent{\bf Result 4.} If $box(G) = \frac{n}{2} - s$, then
$\chi(G) \geq \frac{n}{2s+2}$. (Recall that for a graph $G$ with $n$
vertices $box(G) \leq \frac{n}{2}$.) \\

\noindent {\bf Remark 4:} It should be noted that in general $\chi(G)$ does not seem to have
  much relation with $box(G)$.  There are graphs of very high chromatic number but with very
 low boxicity, for example the complete graphs. Also, there exist graphs of very high boxicity
 but with very low chromatic number, see Section \ref {chromatic_section} for an example. The above Theorem states
 that if the boxicity is \emph {very close} to its maximum achievable value, then the chromatic number
 also has to be \emph {high.}  It may be of interest to the reader to know that recently
 Chandran et. al. \cite {sunil} have shown that for any graph $G$, $box(G) \le 2\chi(G^2)$, where 
 $G^2$ is the square of the graph $G$ i.e.,  the graph obtained by adding edges of the
 form $(u,v)$ to $G$ where $u$ and $v$ were at a distance of exactly $2$ in $G$.

\section{Preliminaries}

Let $G$ be a simple, finite, undirected graph on $n$ vertices. The
vertex set of $G$ is denoted as $V(G)$ and the edge set of $G$ is
denoted as $E(G)$. Let $G'$ be a graph such that $V(G') = V(G)$.
Then, $G'$ is a {\it super graph} of $G$ if $E(G) \subseteq E(G')$.
We define the {\it intersection} of two graphs as follows: if $G_1$
and $G_2$ are two graphs such that $V(G_1) = V(G_2)$, then the
intersection of $G_1$ and $G_2$ denoted as $G = G_1 \cap G_2$ is a
graph with $V(G) = V(G_1) = V(G_2)$ and $E(G) = E(G_1) \cap E(G_2)$.

A set $S \subseteq V(G)$ is called a {\it clique} if $G[S]$, the
induced subgraph of $G$ on $S$, is a complete subgraph of $G$. For a
graph $G$, let $N_G(v)= \{w \in V(G) | vw \in E(G)\}$ be the set of
neighbors of $v$. A cycle on $n$ vertices is denoted as $C_n$. Let
$G$ be a graph. Let $I_1, I_2, \ldots, I_k$ be $k$ interval graphs
(unit interval graphs) such that $G = I_1 \cap I_2 \cap \cdots \cap
I_k$, then $I_1, I_2, \ldots, I_k$ is called an {\it interval graph
representation} ({\it unit interval graph representation}) of $G$.
The following equivalence is well known.

\begin{theo} [\cite{roberts69}] \label{1}
The minimum $k$ such that there exists an interval graph
representation $($unit interval graph representation$)$ of $G$ using
$k$ interval graphs $($unit interval graphs$)$ $I_1, I_2, \ldots,
I_k$ is the same as $box(G)$ $(cub(G))$.
\end{theo}

A graph $G$ is called {\it chordal} if $G$ does not have $C_n$, $n
\geq 4$, as an induced subgraph. Split graphs form a special
subclass of chordal graphs. A graph $G$ is called a {\it split
graph} if $G$ and $\overline{G}$ both are chordal, where
$\overline{G}$ is the complement of the graph $G$. The following
characterization of split graphs is due to F\"oldes {\sl et al.}

\begin{theo}[\cite{hammer}]
$G$ is a split graph if and only if there exists a partition $V = S
\cup K$ of $V(G)$ into an independent set $S$ and a clique $K$.
\end{theo}

In \cite{cozzenrobert}, Cozzens {\sl et al.} studied the boxicity of
split graphs and gave an upper bound.

\begin{theo} [\cite{cozzenrobert}] \label{2.3}
Let $G$ be a split graph with vertex partition $V(G) = S \cup K$,
$S$ an independent set and $K$ a clique. Then provided $K \neq
\emptyset$, $box(G) \leq \min\{\left \lceil\frac{|K|}{2} \right
\rceil, \left \lceil\frac{|S|}{2} \right \rceil\}$.
\end{theo}

\section{Cubicity and vertex cover}

In this section, we give a tight upper bound for cubicity of a graph
$G$ in terms of the cardinality of its minimum vertex cover. In
particular we show that $cub(G) \leq t + \lceil{\log \ (n-t)}\rceil
- 1$, where $|MVC| = t$ and $n$ is the number of vertices of $G$.

Let $MVC = \{v_1, v_2,\ldots, v_t\}$. Clearly $A = V - MVC$ is an
independent set in $G$. Let $A = \{w_0, w_1,\ldots ,w_{\alpha
-1}\}$, where $|A| = n - t = \alpha$. Next, we construct $t +
\left\lceil{\log \ (n-t)}\right\rceil - 1$, unit interval super
graphs of $G$, say $U_1, U_2,\ldots, U_{t+\left\lceil{\log \
(n-t)}\right\rceil - 1}$, as follows.

\noindent{\bf Construction of $U_i$, for $1 \leq i \leq t-1$:} Let
$MVC' = MVC - \{v_t\}$. So, $|MVC'| = t-1$. For each $v_i \in MVC'$,
$1 \leq i \leq t-1$, we construct a unit interval graph $U_i$. To
construct $U_i$, map each $x \in G$ to a unit interval $f_i(x)$ as
follows.

\vspace{-.6cm}

\begin{eqnarray*}
  f_i(x) &=& [0, 1] \ \ \ \ \ \ \ \ \ \ if \ x = v_i  \\
  &=& [1, 2] \ \ \ \ \ \ \ \ \ \ if \ x \in N_G(v_i).\\
         &=& [2, 3] \ \ \ \ \ \ \ \ \ \ if \ x \in V(G) - (N_G(v_i) \cup \{v_i\}).
\end{eqnarray*}

\noindent {\bf Claim:} For each unit interval graph $U_i$, $1 \leq i
\leq t-1$, $E(G) \subseteq E(U_i)$.

\begin{proof}
It is easy to see that for all $x \in N_{G}(v_i) \cup \{v_i\}$, $1
\in f_i(x)$. So, $N_{G}(v_i) \cup \{v_i\}$ induces a clique in
$U_i$. Also, for all $x \in V(G) - \{v_i\}$, $2 \in f_i(x)$. That
is, $V(G) - \{v_i\}$ induces a clique in $U_i$. So, we infer that
$E(G) \subseteq E(U_i)$, for each $i$, $1 \leq i \leq t-1$.
\end{proof}

%For $I_i, 1\leqslant i\leqslant |MVC|-1$, we define the interval
%$[l_{i}(u),r_{i}(u)]$ for each vertex $u\in V(G)$ as follows:  Let
%$l_{i}(v_{i})=0$.  For $v\in V - \{v_{i}\}$, $l_{i}(v)=1$ if
%$(v_{i},v)\in E(G)$ and $l_{i}(v)=2$ otherwise. Note that for each
%vertex $v\in V(G)$, $r_{i}(v)=l_{i}(v)+1$.

\noindent{\bf Construction of $U_{t+j}$, for $0 \leq j \leq
\left\lceil\log \ (n-t)\right\rceil - 1$:} Recall that $MVC = \{v_1,
v_2,\ldots,v_t\}$ and $A = \{w_0, w_1,\ldots,w_{\alpha-1}\}$. It is
easy to see that $v_t$ is adjacent to at least one vertex of $A$
since $MVC$ is a minimum vertex cover of $G$. Without loss of
generality suppose $v_tw_0 \in E(G)$. For each $j$, $0 \leq j \leq
\left\lceil\log \ (n-t)\right\rceil - 1$, we define a function $b_j
: A \longrightarrow \{0,1\}$ as follows:

\vspace{-.6cm}

\begin{eqnarray*}
  b_j(w_k) &=& 0 \ \ \ \ \ if \ the \ (j+1)th \ least \ significant \ bit \ of \ k \ is \ 0\\
   &=& 1 \ \ \ \ \ otherwise.
\end{eqnarray*}

%Consider the binary representation of $w_i$, $0 \leq i \leq \alpha -
%1$. Let $b_j$ be the binary representation of $w_j$ and $b_{j}^{i}$
%correspond to the $i^{th}$ least significant bit of the binary
%representation. Note that $|b_j| = \left\lceil\log \
%(n-t)\right\rceil$. For each bit of $b_j$, we construct a unit
%interval graph as follows.

\noindent To construct $U_{t+j}$, $0 \leq j \leq \left\lceil\log \
(n-t)\right\rceil - 1$, we map each $x \in V(G)$ to a unit interval
as follows.

\vspace{-.6cm}

\begin{eqnarray*}
  f_{t+j}(x) &=& [0.5, 1.5] \ \ \ \ \ \ \ \ \ if \ x = v_t. \\
             &=& [1, 2] \ \ \ \ \ \ \ \ \ \ \ \ \ if \ x \in MVC'. \\
             &=& [0, 1] \ \ \ \ \ \ \ \ \ \ \ \ \ if \ x = w_0. \\
             &=& [0, 1] \ \ \ \ \ \ \ \ \ \ \ \ \ if \ x \in A - \{w_0\} \ and \ b_j(x) = b_j(w_0). \\
             &=& [1.5, 2.5] \ \ \ \ \ \ \ \ if \ x \in A - \{w_0\} \ and \ b_j(x) \neq b_j(w_0) \ and \ xv_t \in E(G). \\
             &=& [2, 3] \ \ \ \ \ \ \ \ \ \ \ \ \ if \ x \in A - \{w_0\} \ and \ b_j(x) \neq b_j(w_0) \
             and \ xv_t \notin E(G).
\end{eqnarray*}

\noindent{\bf Claim 2:} For each unit interval graph $U_{t+j}$, $0
\leq j \leq \left\lceil\log \ (n-t)\right\rceil - 1$, $E(G)
\subseteq E(U_{t+j})$.

\begin{proof}
It is easy to see that, for all $x \in MVC$, $1 \in f_{t+j}(x)$. So,
$MVC$ induces a clique in $U_{t+j}$. Also, for all $y \in N_G(v_t)$,
either $1 \in f_{t+j}(y)$ or $1.5 \in f_{t+j}(y)$. As $f_{t+j}(v_t)
= [0.5, 1.5]$, $f_{t+j}(v_t) \cap f_{t+j}(y) \neq \emptyset$, for
all $y \in N_G(v_t)$. So, $N_G(v_t) \subseteq N_{U_{t+j}}(v_t)$. Let
$w_i \in A$. Now, either $f_{t+j}(w_i) = [0, 1]$ or $[1.5, 2.5]$ or
$[2, 3]$. In all the cases, it is easy to see that $f_{t+j}(w_i)
\cap f_{t+j}(v) \neq \emptyset$, for all $v \in MVC'$ since
$f_{t+j}(v) = [1, 2]$. That is, for each $w_i \in A$, $w_iv \in
E(U_{t+j})$, for all $v \in MVC'$. Hence for each $j$, $0 \leq j
\leq \left\lceil\log \ (n-t)\right\rceil - 1$, $E(G) \subseteq
E(U_{t+j})$.
\end{proof}

The following lemma follows from Claim 1 and Claim 2.

\begin{lemma}
For each unit interval graph $U_i$, $1 \leq i \leq t+\left\lceil\log
\ (n-t)\right\rceil - 1$, $E(G) \subseteq E(U_i)$.
\end{lemma}

%\footnote{LSB stands for Least Significant Bit. For the purpose of
%the proof, we can consider the bits in any arbitrary order}

\begin{lemma}
For any $(x, y) \notin E(G)$, there exists some $i$, $1 \leq i \leq
t+\left\lceil\log \ (n-t)\right\rceil - 1$, such that $(x, y) \notin
E(U_i)$.
\end{lemma}

\begin{proof}
Suppose $(x, y) \notin E(G)$.

\noindent{\bf Case 1:} $\{x, y\} \subseteq MVC$.

It is easy to see that either $x$ or $y$, say $x$, will be present
in $MVC'$. Let $x = v_i$, for some $i$, $1 \leq i \leq t-1$. Now, in
$U_i$, as $y \notin N_G(v_i)$, $f_i(x) = [0, 1]$ and $f_i(y) = [2,
3]$. So, $f_i(x) \cap f_i(y) = \emptyset$. Hence, $x$ is
non-adjacent to $y$ in $U_i$.

\noindent{\bf Case 2:} $x \in MVC$ and $y \in A$.

First suppose $x \in MVC'$. Let $x = v_i$, for some $i$, $1 \leq i
\leq t-1$. Now, in $U_i$, as $y \notin N_G(v_i)$, $f_i(x) = [0, 1]$
and $f_i(y) = [2, 3]$. Hence, $x$ is non-adjacent to $y$ in $U_i$.

Next suppose $x = v_t$. It is easy to see that $y \neq w_0$, as
$w_0v_t \in E(G)$ by assumption. Let $y = w_s$, for some $s$, $1
\leq s \leq \alpha - 1$. Since $s > 0$, clearly there exists a $l$,
$0 \leq l \leq \left\lceil\log \ (n-t)\right\rceil - 1$, such that
$b_l(w_s) \neq b_l(w_0)$. Now, in $U_{t+l}$, $f_{t+l}(w_s) = [2,
3]$. But $f_{t+l}(v_t) = [0.5, 1.5]$. As $f_{t+l}(v_t) \cap
f_{t+l}(w_s) = \emptyset$, $x$ and $y$ are non-adjacent in
$U_{t+l}$.

\noindent{\bf Case 3:} $\{x, y\} \subseteq A$.

Let $x = w_r$ and $y = w_s$, $0 \leq r, s \leq \alpha -1$. Since $r
\neq s$, there exists a $j$, $0 \leq j \leq \left\lceil\log \
(n-t)\right\rceil - 1$, such that $b_j(w_r) \neq b_j(w_s)$. As
$b_j(w_0)$ is either $0$ or $1$, $b_j(w_0)$ is different from either
$b_j(w_r)$ or $b_j(s)$. Without loss of generality let $b_j(w_0)
\neq b_j(w_s)$. So, $b_j(w_0) = b_j(w_r)$ as $b_j(w_r) \neq
b_j(w_s)$. Now, in $U_{t+j}$, $f_{t+j}(w_r) = [0, 1]$ and
$f_{t+j}(w_s) = [1.5, 2.5]$ or $[2, 3]$. In both the cases
$f_{t+j}(w_r) \cap f_{t+j}(w_s) = \emptyset$. Hence $x = w_r$ and $y
= w_s$ are non-adjacent in $U_{t+j}$, $0 \leq l \leq \left\lceil\log
\ (n-t)\right\rceil - 1$.
\end{proof}

%So, $b_r \neq b_s$. Let $b_r^{j} \neq b_s^{j}$, $0 \leq j \leq
%\left\lceil\log \ (n-t)\right\rceil - 1$. Now, The $j^{th}$ bit of
%$w_0$ will either be equal to $b_r^{j}$ or $b_s^{j}$. Suppose
%$b_0^{j} \neq b_s^{j}$. So, $b_0^{j} = b_r^{j}$. Note that $w_0$ and
%$w_r$ may also be same.

By combining the above two lemmas we get $E(G) = E(U_1) \cap E(U_2)
\cap \cdots \cap E(U_{t+\left\lceil\log \ (n-t)\right\rceil - 1})$.
Thus by Theorem \ref{1}, we have the following.

\begin{theo}\label{6.3}
For a graph $G$, $cub(G) \leq t + \lceil{\log \ (n-t)}\rceil - 1$,
where $|MVC| = t$ and $n$ is the number of vertices of $G$.
\end{theo}

\subsection{Tightness result}

In this section we show that the upper bound given for cubicity in
Theorem \ref{6.3} is tight. Let $G$ be a star graph on $n$ vertices.
It is easy to see that $|MVC| = 1$ in $G$. So, $cub(G) \leq 1+\lceil
\log \ (n-1)\rceil-1$ by Theorem \ref{6.3}. That is, $cub(G) \leq
\left\lceil \log\ (n-1)\right\rceil$. But it is known that $cub(G) =
\left\lceil \log\ (n-1)\right\rceil$ \cite{roberts69}. So, the upper
bound for cubicity given in Theorem \ref{6.3} is tight for star
graphs.

\section{Boxicity and vertex cover}

Let $G=(V, E)$ be a graph and $MVC$ be a minimum vertex cover of
$G$. Let $A = V - MVC$. Clearly $A$ is an independent set in $G$.
Suppose $|MVC| = t$ and $ \left \lfloor \frac{t}{2} \right \rfloor =
t_1$.

Let $l$ be the biggest integer such that there exist subsets $P, Q
\subseteq MVC$ such that $P = \{a_1, a_2, \ldots, a_l\}$, $Q =
\{b_1, b_2, \ldots, b_l\}$, $P \cap Q = \emptyset$, and $(a_i, b_i)
\notin E(G)$. Next, we construct $t_1 + 1$ different interval super
graphs of $G$, say $I_1, I_2, \ldots, I_{t_1+1}$, as follows.

\noindent{\bf Construction of $I_i$, for $1 \leq i \leq l$.} Recall
that, for $1 \leq i \leq l$, $(a_i, b_i) \notin E$. For each pair
$(a_i, b_i)$, $1 \leq i \leq l$, we construct an interval graph
$I_i$. To construct $I_i$, we map each $v \in V$ to an interval
$f_i(v)$ on the real line as follows:

\vspace{-.6cm}

\begin{eqnarray*}
f_i(v) &=& [0, 1] \ \ \ \ \ if \ \ v = a_i. \\
       &=& [4, 5] \ \ \ \ \ if \ \ v = b_i. \\
       &=& [0, 3] \ \ \ \ \ if \ \ v \in N_G(a_i) - N_G(b_i). \\
       &=& [2, 5] \ \ \ \ \ if \ \ v \in N_G(b_i) - N_G(a_i). \\
       &=& [0, 5] \ \ \ \ \ if \ \ v \in N_G(a_i) \cap
       N_G(b_i).\\
       &=& [2, 3] \ \ \ \ \ if v \in V - (\{a_i, b_i\} \cup N_G(a_i) \cup
       N_G(b_i)).
  \end{eqnarray*}

%\begin{eqnarray*}
%  if \ \ v&=&a_i, \ then \ f_i(v) = [0,1]. \\
%  if \ \ v&=&b_i, \ then \ f_i(v) = [4,5]. \\
%  if \ \ v&\in&N_{a_i}(G) - N_{b_i}(G), \ then \ f_i(v) = [0,3]. \\
%  if \ \ v&\in&N_{b_i}(G) - N_{a_i}(G), \ then \ f_i(v) = [2,5]. \\
%  if \ \ v&\in&N_{a_i}(G) \cap N_{b_i}(G), \ then \ f_i(v) = [0,5]. \\
%  if \ \ v&\in&V - \{a_i, b_i, N_{a_i}(G) \cup N_{b_i}(G)\}, \ then \ f_i(v) = [3,4].
%  \end{eqnarray*}

%\[ f(x) = \beginfcasesg
%sin(x), & \textf if g 0 \leq x \leq \pi cos(x), & \textf if pi g < x
%\leq 2*\pi
%\endfcasesg \]

%As $P \cap Q = \emptyset$, $(a_i, b_i) \notin E$, for $a_i \in P$
%and $b_i \in Q$. For $1 \leq i \leq l$, we define the interval
%$[l_i(v), r_i(v)]$ for each $v \in V$ as follows: for $a_i$,
%$l_i(a_i) = 0$ and $r_i(a_i) = 1$, for $b_i$, $l_i(b_i) = 4$ and
%$r_i(b_i) = 5$. For $x \in N_{a_i}(G) - N_{b_i}(G)$, $l_i(x) = 0$
%and $r_i(x) = 3$. For $x \in N_{b_i}(G) - N_{a_i}(G)$, $l_i(x) = 2$
%and $r_i(x) = 5$. For $x \in N_{a_i}(G) \cap N_{b_i}(G)$, $l_i(x)
%=0$ and $r_i(x) = 5$. For each $x \in V - \{\{a_i, b_i\} \cup
%\{N_{a_i} \cup N_{b_i}\}\}$, assign $l_i(x) = 2$ and $r_i(x) = 3$.\\

\noindent{\bf Claim 1.} For each $i$, $1 \leq i \leq l$, $E(G)
\subseteq E(I_i)$.

\begin{proof}
It is easy to see that if $v \in MVC - \{a_i,b_i\}$, then $3 \in
f_i(v)$. So, $MVC - \{a_i,b_i\}$ is a clique in each $I_i$. If $v
\in N_G(a_i) \cup \{a_i\}$, then $0 \in f_i(v)$. So, $N_G(a_i)
\subseteq N_{I_i}(a_i)$. Similarly, if $v \in N_G(b_i) \cup
\{b_i\}$, then $5 \in f_i(v)$. That is, $N_G(b_i) \subseteq
N_{I_i}(b_i)$. So, $E(G) \subseteq E(I_i)$, for each $i$, $1 \leq i
\leq l$.
\end{proof}

%\noindent{\bf Claim:} For each $i$, $1 \leq i \leq l$, $I_i$ is an
%interval graph.
%
%Note that each $I_i$ contains 3 maximal cliques, namely $C_1= \{a_i
%\cup N(a_i)\}$, $C_2 = \{b_i \cup N(b_i)\}$, and $\{C\}$. As $a_i$
%and $b_i$ are non-adjacent in $I_i$, $C$ is the only separating
%clique of $I_i$. Now, it is easy to see that the clique tree of
%$I_i$ is a path $C_1, C, C_2$. Hence, $I_i$ is an interval graph for
%each $i$.

\noindent{\bf Construction of $I_i$, for $l+1 \leq i \leq t_1$,}
(assuming $t_1 \geq l+1$). Let $C = MVC - \{P \cup Q\}$. Clearly $C$
induces a clique in $G$ by the maximality of $l$. Let $|C| = k' = t
- 2l$. Since $t_1 = \left\lfloor\frac{t}{2}\right\rfloor$, we have
$k' = t - 2l \geq 2$ and $C \geq 2$. Let $C =
\{c_1,c_2,\ldots,c_{k'}\}$. If $k'$ is even, then let $k'' = k'$,
otherwise let $k'' = k' - 1$. Let $C' = \{c_1, c_2, \ldots,
c_{k''}\}$. Clearly $C' \subseteq C$.

Let $G'$ be the graph induced by $C' \cup A$ in $G$. As $C'$ induces
a clique and $A$ induces an independent set in $G$, $G'$ is a split
graph. So by Theorem \ref{2.3}, \ $box(G') \leq$ min$\{ \left \lceil
\frac{k''}{2} \right \rceil, \left\lceil \frac{|A|}{2}\right\rceil\}
\leq \frac{k''}{2}$ (as $k''$ is even and $k'' \geq 2$). That is,
$G'$ is the intersection of at most $\frac{k''}{2}$ interval graphs,
say $I'_1, I'_2, \ldots, I'_{\frac{k''}{2}}$, by Theorem \ref{1}.
Note that $l + \frac{k''}{2} = \left\lfloor\frac{t}{2}\right\rfloor
= t_1$. Let $g_i$ be a function that maps each $v \in V(I'_i)$ to a
closed interval on the real line such that $I'_i$, for each $i$, $1
\leq i \leq \frac{k''}{2}$, is the intersection graph of the family
of intervals $\{g_i(v) : v \in V(I'_i)\}$. Now, let $L_j$ and $R_j$
be numbers on the real line such that $L_j \leq x$, for all $x \in
\bigcup_{v \in V(I'_j)}(g_j(v))$ and $R_j \geq y$, for all $y \in
\bigcup_{v \in V(I'_j)}(g_j(v))$. To construct $I_i$, $l+1 \leq i
\leq t_1$, map each $v \in V(G)$ to a closed interval $f_{l+j}(v)$,
$1 \leq j \leq \frac{k''}{2}$ on the real line as follows.

\vspace{-.6cm}

\begin{eqnarray*}
  f_{l+j}(v) &=& g_j(v) \ \ \ \ \ \ if \ v \in V(I'_j) = V(G)-(P \cup Q)-(C-C')  \\
         &=& [L_j, R_j] \ \ \ \ otherwise.
\end{eqnarray*}

%For each $x \in P \cup Q$, put the interval $[L_i, R_i]$ in $I'_i$,
%$1 \leq i \leq \frac{k''}{2}$. After this construction let the
%interval graphs be $I_{l+1}, I_{l+2}, \ldots, I_{l+\frac{k''}{2}}$.
%Note that $l + \frac{k''}{2}$ is nothing but $t_1$. Clearly for each
%$i$, $V(I_i)
%= V(G)$, $l+1 \leq i \leq t_1$.\\

\noindent{\bf Claim 2.} For each $I_i$, $l+1 \leq i \leq t_1$, $E(G)
\subseteq E(I_i)$.

\begin{proof}
By the construction of $I_i$, $l+1 \leq i \leq t_1$, it is easy to
see that if $v \in P \cup Q \cup (C-C')$, then $L_j \in f_{l+j}(v)$,
$1 \leq j \leq \frac{k''}{2}$. So, $P \cup Q \cup (C-C')$ induces a
clique in each $I_i$, $l+1 \leq i \leq t_1$. Also, if $u \in P \cup
Q \cup (C-C')$, then $uv \in E(I_i)$, for each $v \in V(I_i) - \{P
\cup Q \cup (C-C')\}$, by the definition of $L_i$ and $R_i$. As the
collection of interval graphs $I'_1, I'_2,\ldots,I'_{\frac{k''}{2}}$
is an interval graph representation of $G'$, by Theorem \ref{1},
$E(G') \subseteq E(I'_j)$, $1 \leq j \leq \frac{k''}{2}$. But in
$I_{l+j}$, $f_{l+j}(v) = g_j(v)$, for all $v \in V(I'_j)$, $1 \leq j
\leq \frac{k''}{2}$. So, $E(G') \subseteq E(I_{l+j})$, $1 \leq j
\leq \frac{k''}{2}$. Hence for each $I_i$, $l+1 \leq i \leq t_1$,
$E(G) \subseteq E(I_i)$.
\end{proof}

\noindent{\bf Construction of $I_{t_1 + 1}$.}  We construct the last
interval graph $I_{t_1+1}$ as follows. If $k'$ is odd then suppose
$C - C' = \{v\}$. So, $v \notin V(G')$. Let $MVC' = MVC$ if $k'$ is
even and $MVC'=MVC - \{v\}$ if $k'$ is odd. Let $A = \{x_1, x_2,
\ldots, x_r\}$, where $|A| = r$. Note that $A \neq \emptyset$. If
$k'$ is odd, then without loss of generality suppose $\{x_1,
x_2,\ldots,x_s\} = A \cap N_G(v)$. Now, map each vertex $x$ of $G$
to an interval $f_{t_1+1}(x)$ on the real line as follows.

\vspace{-.6cm}

 \begin{eqnarray*}
  f_{t_1+1}(x) &=&  [2i - 1, 2i] \ \ \ \ \ \ if \ x \in A \ and \ x=x_i.\\
       &=&  [1, 2r] \ \ \ \ \ \ \ \ \ \ \ \ if \ x \in MVC'. \\
       if \ k' \ is \ odd \ then \ f_{t_1+1}(v) &=&  [1, 2s] \ \ \ \ \ \ \ \ \ \ \ \
 \end{eqnarray*}

\noindent{\bf Claim 3.} $E(G) \subseteq E(I_{t_1+1})$.

\begin{proof}
It is easy to see that if $x \in MVC$, then $1 \in f_{t_1+1}(x)$.
So, $MVC$ induces a clique in $I_{t_1+1}$. Also, if $x \in MVC' \cup
\{x_i\}$, for some $i$, $1 \leq i \leq r$, then $2i \in
f_{t_1+1}(x)$. That is, each $x_i \in A$ is adjacent to all the
vertices of $MVC'$. If $x = x_i$, $1 \leq i \leq s$, then $2i \in
f_{t_1+1}(a_i) \cap f_{t_1+1}(v) \neq \emptyset$. Thus $(x_i, v) \in
E(I_{t_1+1})$ for each $i$, $1 \leq i \leq s$. That is, $N_{G}(v)
\subseteq N_{I_{t_1+1}}(v)$. So, $E(G) \subseteq E(I_{t_1+1})$.
\end{proof}

\noindent The following lemma follows from Claim 1, Claim 2, and
Claim 3.

\begin{lemma}
For each interval graph $I_i$, $1 \leq i \leq t_1+1$, $E(G)
\subseteq E(I_i)$.
\end{lemma}

\begin{lemma}
For any $(x, y) \notin E(G)$, there exists some $i$, $1 \leq i \leq
t_1+1$, such that $(x,y) \notin E(I_i)$.
\end{lemma}

\begin{proof}
Suppose $(x,y) \notin E(G)$. As $C$ induces a clique in $G$, both
$x$ and $y$ cannot be present in $C$.

\noindent{\bf Case 1:} $\{x, y\} \subseteq A$.

Let $x = x_i$ and $y = x_j$, where $i \neq j$. It is easy to see
that $f_{t_1+1}(x) \cap f_{t_1+1}(y) = \emptyset$. Hence $x$ is
non-adjacent to $y$ in $I_{t_1+1}$.

\noindent{\bf Case 2:} $\{x, y\} \cap \{P\cup Q\} \neq \emptyset$

Without loss of generality suppose $x \in P \cup Q$. So, in $I_i$,
for some $i$, $1 \leq i \leq l$, say $I_k$, $f_k(x) = [0, 1]$ or
$f_k(x) = [4, 5]$. If $f_k(x) = [0, 1]$, then $f_k(y)$ is either
$[2, 3]$, $[2, 5]$ or $[4, 5]$ and if $f_k(x) = [4, 5]$, then
$f_k(y)$ is either $[0, 1]$, $[0, 3]$ or $[2, 3]$. In both the cases
$f_k(x) \cap f_k(y) = \emptyset$. Hence $x$ is non-adjacent to $y$
in $I_k$.

\noindent{\bf Case 3:} $\{x, y\} \cap \{P\cup Q\} = \emptyset$

Now, it is easy to see that one of $x$ or $y$, say $x$, will belong
to $MVC - \{P \cup Q\}$, and $y$ will belong to $A$. If $x \in C'$,
then it is easy to see that $x, y \in V(G')$. As $I'_1,
I'_2,\ldots,I'_{\frac{k''}{2}}$ is an interval graph representation
of $G'$, by Theorem \ref{1}, there exists $k$, $1 \leq k \leq
\frac{k''}{2}$ such that $(x, y) \notin I'_k$. But in $I_{l+k}$,
$f_{l+k}(v) = g_k(v)$, for all $v \in I'_k$. So, $x$ and $y$ are
non-adjacent in $I_{l+k}$.

Next suppose $x \in C-C'$. Now, in $I_{t_1+1}$, $f_{t_1+1}(x) = [1,
2s]$ and as $y \notin N_x(G)$, $y = c_j$, where $j > s$. It is easy
to see that $f_{t_1+1}(x) \cap f_{t_1+1}(y) = \emptyset$. So, $x$
and $y$ are non-adjacent in $I_{t_1+1}$.

Hence there exists some $i$, $1 \leq i \leq t_1 + 1$, such that
$(x,y) \notin E(I_i)$.
\end{proof}

%Next, if $x$ and $y$ both belong to $MVC$, then in some $I_i$, $1
%\leq i \leq l$, say $I_k$, $f_k(x) = [0, 1]$ and $f_k(y) = [4, 5]$.
%Thus $x$ and $y$ are non-adjacent in $I_k$, $1 \leq k \leq l$. So,
%wlg suppose $x \in MVC$ and $y \in A$. If $x$ is the vertex $v$ (as
%described in the construction of $I_{t_1+1}$), then $f_{t_1+1}(x) =
%[1, 2r]$ and $f_{t_1+1} = [2j-1, 2j]$, for $r < j \leq s$ (as $v$
%and $y$ are non-adjacent). As $f_{t_1+1}(x)$ and $f_{t_1+1}(y)$ do
%not intersect, $x$ and $y$ are non-adjacent in $I_{t_1+1}$.

%If $x \notin C$, then in some $I_i$, $1 \leq i \leq l$, say $I_k$,
%$f_k(x)$ is either $[0, 1]$ or $[4, 5]$ and $f_k(y) = [2, 3]$. In
%both the cases $f_k(x)$ and $f_k(y)$ do not intersect. So, $x$ and
%$y$ are non-adjacent in $I_k$, for some $k$, $1 \leq k \leq l$.

By combining the above two lemmas we get $E(G) = E(I_1) \cap E(I_2)
\cap \cdots \cap E(I_{t_1+1})$. Thus by Theorem \ref{1}, we obtain
the following.

\begin{theo} \label{3.3}
For a graph $G$ with vertex cover $MVC$, $box(G) \leq \lfloor
\frac{t}{2}\rfloor + 1$, where $t = |MVC|$.
\end{theo}

%\begin{proof}
%Let $G$ be a graph and $vc$ be a vertex cover of $G$. Let $|vc| = k$
%and $v_1, v_2, \ldots, v_k$ be the vertices present in $vc$. As $G$
%has a vertex cover of size $k$, $G$ has an independent set, say $A$,
%of size $n-k$.
%
%
%
%We will construct $\lceil\frac{k}{2}\rceil + 1$ ($= k_1$ say)
%interval graphs, say $I_i$, $1 \leq i \leq k_1$, and will show that
%$E(G) = \bigcap_{i=1}^{k_1}{I_i}$.
%
%For each non-adjacent pair $a_i - b_i$ of vertices in $vc(G)$, an
%interval graph is constructed as follows:
%
%
%Construct a clique $C$ of size $n-2$. Take two non-adjacent vertices
%$a_i$ and $b_i$ of $G$ which are present in $vc(G)$. Make $a_i$
%adjacent to the vertices $N(a_i)$ and make $b_i$ adjacent to the
%vertices $N(b_i)$ in $G$. Note that, $C \supseteq N(a_i) \cup
%N(b_i)$. Clearly $G_i$ is a supergraph of $G$.
%
%\noindent{\bf Claim:} $G_i$ is an interval graph.
%
%$G_i$ contains 3 maximal cliques, namely $C_1= \{a_i \cup N(a_i)\}$,
%$C_2 = \{b_i \cup N(b_i)\}$ and  $\{C\}$. By construction of $G_i$,
%$C$ is the only separating clique of $G_i$. Now, it is easy to see
%that the clique tree of $G_i$ is a path $C_1, C, C_2$. Hence, $G_i$
%is an interval graph.
%
%Suppose $vc(G)$ has $i$ pair wise non-adjacent vertices. So, by
%removing these $2i$ vertices from $vc$, rest $k - 2i$ vertices will
%form a clique. Let $G'$ be the graph after removing those $2i$
%vertices. Now, it is easy to see that $G'$ is a split graph.
%
%
%
%\end{proof}

\subsection{Tightness result}

In this section we illustrate some graphs for which the bound given
in Theorem \ref{3.3} for boxicity is tight. Consider the graph
$C_4$, a cycle of length four. The size of minimum vertex cover of
$C_4$ is 2. It is easy to verify that the boxicity of $C_4$ is two.
So, $box(C_4) = \frac{|MVC|}{2} + 1$.

Roberts has shown that for any even number $n$, there exists a graph
on $n$ vertices with boxicity $\frac{n}{2}$. Such graphs are called
Roberts graphs. The Roberts graph on $n$ vertices is obtained by
removing the edges of a perfect
matching from the complete graph $K_n$.\\

\noindent {\bf Claim:} For Roberts graph $G$ on $n$ vertices, the
cardinality of minimum vertex cover is $n - 2$.

\begin{proof}
Let $a, b \in V(G)$ be such that $(a, b) \notin E(G)$. It is easy to
verify that $V - \{a, b\}$ is a vertex cover of $G$. Thus, $|MVC|
\leq n - 2$. Now, if possible suppose $|MVC| \leq n - 3$. Let $a$,
$b$, and $c$ be the vertices which are not present in $MVC$. By the
construction of Roberts graph there will exist an edge in the
induced subgraph of $G$ on $\{a, b, c\}$. Clearly this edge is not
adjacent to any of the vertex of $MVC$. This is a contradiction.
Hence for Roberts graph on $n$ vertices $|MVC| = n -2$.
\end{proof}

For Roberts graphs $\left\lfloor\frac{|MVC|}{2}\right\rfloor + 1$ =
$\left\lfloor\frac{n-2}{2}\right\rfloor + 1$ = $\frac{n}{2}$ (as $n$
is even), which equals the boxicity of Roberts graph. Thus the bound
of Theorem \ref{3.3} is tight for Roberts graphs.

\section{Boxicity and bipartite graphs}

Let $G = (V_1\cup V_2, E)$ be a bipartite graph such that $|V_1| =
n_1$ and $|V_2| = n_2$. Suppose $n_1 \leq n_2$ and $n_1 \geq 3$. In
this section we show that for a bipartite graph $G$, $box(G) \leq $
min $\{\lceil\frac{n_1}{2}\rceil, \lceil\frac{n_2}{2}\rceil\}$.

It is easy to see that $|MVC| \leq n_1$ in $G$. So, by Theorem
\ref{3.3}, $box(G) \leq \lfloor\frac{n_1}{2}\rfloor + 1$. If $n_1$
is odd, then $\left\lfloor\frac{n_1}{2}\right\rfloor + 1 =
\left\lceil \frac{n_1}{2} \right\rceil$. So, $box(G) \leq $ min
$\{\left\lceil\frac{n_1}{2}\right\rceil,
\left\lceil\frac{n_2}{2}\right\rceil\}$.

Now assume that $n_1$ is even. By Theorem \ref{3.3}, $box(G) \leq
\left\lfloor\frac{n_1}{2}\right\rfloor + 1 = \frac{n_1}{2} + 1$ (as
$n_1$ is even). But, we need to show that $box(G) \leq
\frac{n_1}{2}$. So that, $box(G) \leq
\min\{\left\lceil\frac{n_1}{2}\right\rceil,
\left\lceil\frac{n_2}{2}\right\rceil\}$.

Suppose $n_1$ is even. We construct $\frac{n_1}{2}$ interval super
graphs of $G$, say $I_1, I_2,\ldots,I_{\frac{n_1}{2}}$, as follows.

\noindent{\bf Construction of $I_i$, for $1 \leq i \leq
\frac{n_1}{2} - 1$:}  Let $x,y \in V_1$ and $V_1' = V_1 - \{x,y\}$.
Note that $V_1' \neq \emptyset$ as $|V_1| \geq 3$. Let $G_1'$ be the
graph induced by $V_1' \cup V_2$ in $G$. Let $G_1$ be a graph such
that $V(G_1) = V(G_1')$ and $E(G_1) = E(G_1') \cup \{xy \ | \ x, y
\in V_1'\}$. Clearly $V_1'$ induces a clique and $V_2$ induces an
independent set in $G_1$. So, $G_1$ is a split graph. Now, by
Theorem \ref{2.3}, $box(G_1) \leq$
min$\{\left\lceil\frac{n_1-2}{2}\right\rceil,
\left\lceil\frac{n_2}{2}\right\rceil\} =
\left\lceil\frac{n_1-2}{2}\right\rceil = \frac{n_1}{2} - 1$ (as
$n_1$ is even). That is, $G_1$ is the intersection of at most
$\frac{n_1}{2} - 1$ interval graphs, say $I'_1, I'_2, \ldots,
I'_{\frac{n_1}{2} - 1}$, by Theorem \ref{1}.

Let $h_i$ be a function that maps each $v \in V(I'_i)$, $1 \leq i
\leq \frac{n_1}{2} - 1$, to a closed interval on the real line such
that $I'_i$ is the intersection graph of the family of intervals
$\{h_i(v) : v \in V(I'_i)\}$. Now, let $L'_i$ and $R'_i$, $1 \leq i
\leq \frac{n_1}{2} - 1$, be numbers on the real line such that $L'_i
\leq x$, for all $x \in \bigcup_{v \in V(I'_i)}(h_i(v))$ and $R'_i
\geq y$, for all $y \in \bigcup_{v \in V(I'_i)}(h_i(v))$. To
construct $I_i$, $1 \leq i \leq \frac{n_1}{2} - 1$, map each $v \in
V(G)$ to a closed interval $f_{i}(v)$ on the real line as follows.

\vspace{-.6cm}

\begin{eqnarray*}
  f_{i}(v) &=& h_i(v) \ \ \ \ \ \ if \ v \in V(I'_i) = V(G) - \{x,y\}. \\
         &=& [L'_i, R'_i] \ \ \ \ otherwise.
\end{eqnarray*}

\noindent{\bf Claim 1:} For each $I_i$, $1 \leq i \leq \frac{n_1}{2}
- 1$, $E(G) \subseteq E(I_i)$.

\begin{proof}
Since $f_i(x) = f_i(y) = [L'_i, R'_i]$ it is easy to see that in  $I_i$,
$x$ and $y$ are adjacent to each $v$, $v \in
V(I_i) - \{x,y\}$, by the definition of $L'_i$ and $R'_i$. As $I'_1,
I'_2,\ldots,I'_{\frac{n_1}{2}-1}$ is an interval graph
representation of $G_1$, by Theorem \ref{1}, $E(G_1) \subseteq
E(I'_i)$, for each $1 \leq i \leq \frac{n_1}{2} - 1$. But in $I_i$,
$f_i(v) = h_i(v)$, for all $v \in V(I'_i)$. So, $E(G_1) \subseteq
E(I_i)$, $1 \leq i \leq \frac{n_1}{2} - 1$. Hence for each $I_i$, $1
\leq i \leq \frac{n_1}{2} - 1$, $E(G) \subseteq E(I_i)$.
\end{proof}

\noindent{\bf Construction of $I_{\frac{n_1}{2}}$:} Let $V_1 =
\{v_1, v_2, \ldots, v_{n_1}\}$. Suppose without loss of generality that $x = v_1$ and $y = v_{n_1}$.
To construct $I_{\frac{n_1}{2}}$, we map each $v \in V(G)$ to an
interval $f_{\frac{n_1}{2}}(v)$ as follows.

\vspace{-.6cm}

\begin{eqnarray*}
  f_{\frac{n_1}{2}}(v)  &=& [2i - 1, 2i] \ \ \ \ \ \ \ \ \ if \ v \in V_1 \ and \ v = v_i\\
                        &=& [1, 2n_1] \ \ \ \ \ \ \ \ \ \ \ \ \ \mbox { if $v \in V_2$ and }  \ v \in N_x \cap
                        N_y. \\
                        &=& [1, 2n_1 - 2] \ \ \ \ \ \ \mbox { if $v \in V_2$ and }  \ v \in N_x -
                        N_y. \\
                        &=& [3, 2n_1] \ \ \ \ \ \ \ \ \ \ \ \ \ \ \mbox { if $v \in V_2$ and } \ v \in N_y
                        -
                        N_x. \\
                        &=& [3, 2n_1 - 2] \ \ \ \ \ \ \  if   \ v \in V_2 - (N_x
                        \cup N_y).
\end{eqnarray*}

%Now, for each $v_i \in V_1$, $1 \leq i \leq n_1$, assign an interval
%$[l(v_i), r(v_i)]$ as $l(v_i) = 2i - 1$ and $r(v_i) = 2i$. For each
%$v \in V_2$, if $v$ is adjacent to both $x$ and $y$, then $l(v) = 1$
%and $r(v) = 2n_1$. If $v$ is adjacent to $x$ but not adjacent to $y$
%then $l(v) = 1$ and $r(v) = 2(n_1-2)$. If $v$ is adjacent to $y$ but
%not adjacent to $x$ then $l(v) = 3$ and $r(v) = 2n_1$. If $v$ is not
%adjacent to $x$ and $y$ then $l(v) = 3$ and $r(v) = 2(n_1-2)$.\\

\noindent{\bf Claim 2:} $E(G) \subseteq E(I_{\frac{n_1}{2}})$.

\begin{proof}
In $I_{\frac{n_1}{2}}$, for each $v \in V_2$, the point $n_1 \in
f_{\frac{n_1}{2}}(v)$. So, $V_2$ induces a clique in
$I_{\frac{n_1}{2}}$. Also for each $v \in N_G(x)$, $1 \in
f_{\frac{n_1}{2}}(v)$ and for each $v \in N_G(y)$, $2n_1 \in
f_{\frac{n_1}{2}}(v)$. So, $N_G(x) \subseteq
N_{I_{\frac{n_1}{2}}}(x)$ and $N_G(y) \subseteq
N_{I_{\frac{n_1}{2}}}(y)$. For $v_j \in V_1 - \{x,y\}$, we have $2
\leq j \leq n_1-1$, and thus we have $3 \leq 2j-1 \leq 2n_1-2$. So,
$2j-1 \in f_{\frac{n_1}{2}}(v)$, for all $v \in V_2$. It is easy to
see that 
$(v_i,v) \in E(I_{\frac{n_1}{2}})$ for all pairs $(v_i,v)$ where 
$v_i \in V_1 - \{x, y\}$ and  
 $v \in V_2$. Hence $E(G) \subseteq
E(I_{\frac{n_1}{2}})$.
\end{proof}

\noindent The following lemma follows from Claim 1 and Claim 2.

\begin{lemma}
For each interval graph $I_i$, $1 \leq i \leq \frac{n_1}{2}$, $E(G)
\subseteq E(I_i)$.
\end{lemma}

\begin{lemma}
For any $(p, q) \notin E(G)$, there exists some $i$, $1 \leq i \leq
\frac{n_1}{2}$, such that $(p, q) \notin E(I_i)$.
\end{lemma}

\begin{proof}
Suppose $(p, q) \notin E(G)$.

\noindent{\bf Case 1:} $\{p, q\} \subseteq V_1$.

If both $p$ and $q$ belong to $V_1$, then suppose $p = v_i$ and $q =
v_j$, where $i \neq j$. In this case it is easy to see that in
$I_{\frac{n_1}{2}}$, $f_{\frac{n_1}{2}}(p) \cap f_{\frac{n_1}{2}}(q)
= \emptyset$. So, $p$ is non-adjacent to $q$ in $I_{\frac{n_1}{2}}$.

\noindent{\bf Case 2:} $\{p, q\} \subseteq V_2 \cup V_1'$.

If both $p$ and $q$ belong to $V_2 \cup V'_1$, then it is easy to see that $p,
q \in G_1$ and in view of case 1, $(p,q) \notin E(G_1)$.
 As $I'_1, I'_2,\ldots, I'_{\frac{n_1}{2} - 1}$ is an
interval graph representation of $G_1$, by Theorem \ref{1}, $(p, q)
\notin E(I'_i)$, for some $i$, $1 \leq i \leq \frac{n_1}{2} - 1$,
say $I'_k$. Recalling that in $I_k$, 
 $f_{k}(v) = g_k(v)$ for all $v \in V(I'_k)$
$p$ and $q$ are non-adjacent in $I_k$ also.

\noindent{\bf Case 3:} $p \in \{x,y\}$ and $q \in V_2$.

Let $p = x$. Now, in $I_{\frac{n_1}{2}}$, $f_{\frac{n_1}{2}}(x) =
[1, 2]$ and as $q$ is not a neighbor of $x$ in $G$, either
$f_{\frac{n_1}{2}}(q) = [3, 2n_1-2]$ or $[3, 2n_1]$. In both the
cases, $f_{\frac{n_1}{2}}(p) \cap f_{\frac{n_1}{2}}(q) = \emptyset$.
So, $p$ and $q$ are non-adjacent in $I_{\frac{n_1}{2}}$.

Similarly, if $p = y$, then in in $I_{\frac{n_1}{2}}$,
$f_{\frac{n_1}{2}}(y) = [2n_1-1, 2n_1]$ and as $q$ is not a
neighbor of $y$ in $G$, either $f_{\frac{n_1}{2}}(q) = [1,
2n_1-2]$ or $[3, 2n_1-2]$. In both the cases, $f_{\frac{n_1}{2}}(p)
\cap f_{\frac{n_1}{2}}(q) = \emptyset$. So, $p$ and $q$ are
non-adjacent in $I_{\frac{n_1}{2}}$.
\end{proof}

%Hence there exists some $i$, $1 \leq i \leq \frac{n_1}{2}$, such
%that $(p, q) \notin E(I_i)$.

By combining the above two lemmas we get $E(G) = E(I_1) \cap E(I_2)
\cap \cdots \cap E(I_{\frac{n_1}{2}})$. Thus by Theorem \ref{1}, we
have the following.

\begin{theo}\label{4.3}
For a bipartite graph $G= (V_1 \cup V_2, E)$, $box(G) \leq$ min
$\{\left\lceil \frac{|n_1|}{2}\right\rceil, \left\lceil
\frac{|n_2|}{2}\right\rceil\}$, where $|V_1| = n_1$ and $V_2 = n_2$.
\end{theo}

{
\subsection{Tightness result}
\label {chromatic_section}
}

In this section we show that the bound given in Theorem \ref{4.3} is
tight. Consider a complete bipartite graph $G = (V_1 \cup V_2, E)$ 
where $|V(G)| = n$ and $|V_1|=|V_2|= \frac {n}{2}$.
Now remove a perfect matching $M$  from that. Let $G' = (V_1 \cup V_2,
E')$ be the resulting graph.  We show that
$box(G') = \lceil\frac{n}{4}\rceil$.

\vspace{.2cm}

\noindent{\bf Claim:} $box(G') \geq
\left\lceil\frac{n}{4}\right\rceil$.

\begin{proof}
If possible suppose $box(G') \leq \left\lceil\frac{n}{4}\right\rceil
- 1$. Let $n$ be divisible by $4$. So,
$\left\lceil\frac{n}{4}\right\rceil - 1 = \frac{n}{4} - 1$, By
Theorem \ref{1}, $G'$ is the intersection of at most $\frac{n}{4} -
1$ interval graphs, sat $I_1,I_2,\ldots,I_{\frac {n}{4} -1}$.
Recall that $M \not \subseteq E(G')$ and thus for each $e \in M$
there exists a $k$, $1 \le k \le \frac {n}{4} -1$, 
such that $e \notin E(I_k)$, since $G'$ is the
intersection of the  interval graphs $I_1, I_2, \ldots, I_{\frac {n}{4} -1}$.
Since $|M| = \frac {n}{2}$, by pigeon hole principle  we infer that there exists a $j$, 
$1 \le j \le \frac {n}{4} - 1$, such that 
  at least three edges of $M$ are missing in $E(I_j)$.
Now, it is easy to see that $I_k$ contains an induced cycle of
length six, a contradiction.

Next suppose $n$ is not divisible by $4$. Now,
$\left\lceil\frac{n}{4}\right\rceil - 1 =
\left\lfloor\frac{n}{4}\right\rfloor$. Using similar arguments as 
when $n$ is divisible by $4$, we will get a contradiction
in this case also.

Hence, $box(G') \geq \left\lceil\frac{n}{4}\right\rceil$.
\end{proof}

By Theorem \ref{4.3}, we have $box(G') \leq
\lceil\frac{n}{4}\rceil$. Hence $box(G') = \lceil\frac{n}{4}\rceil$.
It follows that the bound of Theorem \ref {4.3} is tight.

\section{Boxicity and chromatic number}

%In this section, we show that when boxicity of a graph is near to
%$\frac{n}{2}$, then chromatic number become very high, i.e., near to
%$\frac{n}{2}$. Let $\chi$ and $\alpha$ be the chromatic number and
%independence number of a graph $G$.

We know that $box(G) \leq \left\lfloor\frac{n}{2}\right\rfloor$,
where $n$ is the number of vertices of $G$ \cite{roberts69}. Let
$box(G) = \frac{n}{2} - s$, for some $s \geq 0$. Note that, if $n$
is odd, then $s$ is not an integer. In the following theorem, we
show that when $s$ is small for a graph $G$, the chromatic number of
$G$ has to be very high.

\begin{theo}\label{5.1}
If $box(G) = \frac{n}{2} - s$, then $\chi(G) \geq \frac{n}{2s+2}$.
\end{theo}

\begin{proof}
Let $box(G) = \frac{n}{2} - s$. By Theorem \ref{3.3}, $box(G) =
\frac{n}{2} - s \leq \left\lfloor\frac{t}{2}\right\rfloor + 1 \leq
\frac{t}{2} + 1$, where $t$ is the cardinality of a minimum vertex
cover of $G$. So, $t \geq n - 2s - 2$. It is easy to see that if
$\alpha$ is the independence number of $G$, then $\chi \geq
\frac{n}{\alpha}$. But $\alpha = n - t$. So,

\vspace{-.6cm}

\begin{eqnarray*}
  \chi(G) &\geq& \frac{n}{n - t} \\
  ~ &\geq& \frac{n}{n-(n-2s-2)} \\
  ~ &=& \frac{n}{2s+2}
\end{eqnarray*}
%So, $\chi(G) \geq \frac{n}{2s+2}$.

\end{proof}

\noindent{\bf Remark:} The lower bound for $\chi(G)$ given in
Theorem \ref{5.1} is tight in the case of Roberts graphs. We know
that if $G$ is a Roberts graph on $n$ vertices, then $box(G) =
\frac{n}{2}$ (recall that $n$ is even for Roberts graph). Thus, $s =
0$ for $G$. Putting the value of $s$ in the inequality given in
Theorem \ref{5.1}, we get $\chi(G) \geq \frac{n}{2}$. But it is easy
to verify that $\chi(G) = \frac{n}{2}$.

%\bibliographystyle{plain}
%\bibliography{references}

\end{document}